\documentclass[12pt]{amsart}

\usepackage{amssymb,amsmath,graphics,verbatim}
\usepackage{latexsym}
\usepackage{eucal}
\usepackage{a4wide}

\newtheorem{theorem}{Theorem}
\newtheorem{lemma}[theorem]{Lemma}

\newtheorem{corollary}[theorem]{Corollary}
\theoremstyle{definition}
\newtheorem{definition}[theorem]{Definition}
\theoremstyle{remark}

\newtheorem{remark}[theorem]{Remark}

\newcommand{\cC}{\mathcal{C}}
\newcommand{\cM}{\mathcal{M}}
\newcommand{\cpu}{\cz \mathrm{P}^1}
\newcommand{\cun}{\cC^\infty}
\newcommand{\cz}{\mathbb{C}}
\newcommand{\fu}{{1_\Sigma}}
\newcommand{\nz}{\mathbb{N}}
\newcommand{\pr}{\partial_r}
\newcommand{\Res}{\mathrm{Res}}
\newcommand{\rpd}{\rz \!\mathrm{P}^2}
\newcommand{\rz}{\mathbb{R}}
\newcommand{\tz}{\mathbb{T}}
\newcommand{\zz}{\mathbb{Z}}

\begin{document}
\title{Uniformization of $S^2$ and flat singular surfaces}
\date{\today}
\author{Sergiu Moroianu}
\thanks{Partially supported by grant PN-II-ID-PCE 1188 265/2009}
\address{Institutul de Matematic\u{a} al Academiei Rom\^{a}ne\\
P.O. Box 1-764\\
RO-014700 Bucharest, Romania}
\email{moroianu@alum.mit.edu}
\subjclass[2000]{Primary: 58J60}

\begin{abstract}
We construct flat metrics in a given conformal class with prescribed singularities of real orders at marked points of a closed real surface. The singularities can be small conical, cylindrical, and large conical with possible translation component. Along these lines we give an elementary proof of the uniformization theorem for the sphere.
\end{abstract}

\maketitle

\section{Uniformization in genus $0$ and holomorphic structures on conformal surfaces}

Let $(\Sigma,g)$ be a closed, possibly unorientable Riemannian manifold of dimension $2$. For background material on surfaces we refer the reader to \cite{M}.
The Gaussian curvature $\kappa_g:\Sigma\to\rz$ of the metric $g$ is the sectional curvature function of the tangent planes to $\Sigma$. If $f\in\cun(\Sigma,\rz)$ is any smooth function, an elementary computation shows that the Gaussian curvatures of the conformal metrics $g$ and $g':=e^{-2f}g$ are related by
\begin{equation}\label{ccc}
\kappa_{g'}=e^{2f}(\kappa_g-\Delta_g f)
\end{equation}
where $\Delta_g=d^*d$ is the Laplacian on functions with respect to the metric $g$.
By the Gauss-Bonnet theorem, 
\[\int_\Sigma \kappa_g |\mu_g|=2\pi\chi(\Sigma)\]
where $\chi(\Sigma)$ is the Euler characteristic of $\Sigma$, while $|\mu_g|$ is the volume density.
On the other hand, denoting by $\langle\cdot,\cdot\rangle$ the scalar product in $L^2(\Sigma,|\mu_g|)$ and by $\fu$ the constant function $1$ on $\Sigma$, we have
\[\int_\Sigma \Delta f|\mu_g|=\langle \Delta f,\fu\rangle=\langle d f,d\fu\rangle=0.\]
Thus in the conformal class of $g$ there exist flat metrics only if $\chi(\Sigma)=0$, i.e., $\Sigma$ is a torus or a Klein bottle.

\begin{theorem}\label{genus0}
Let $(\Sigma,g)$ be a possibly unorientable closed Riemannian manifold of dimension $2$. In the conformal class of $g$ there exist flat metrics (and then they are unique up to dilations) if and only if $\chi(\Sigma)=0$. 
\end{theorem}
\begin{proof}
The necessity of $\chi(\Sigma)=0$ was seen above, so assume $\chi(\Sigma)=0$. The Gauss-Bonnet theorem says  that the function $\kappa_g$ is of average $0$, or equivalently that it is orthogonal in $L^2$ sense on the constant function $1$. But $\Delta$ is self-adjoint and Fredholm in $L^2$, so 
its image is the orthogonal complement of its kernel. Moreover, $\phi\in\ker(\Delta)$ if and only if $\phi$ is constant. Thus $\kappa_g$ lives in the image of the Laplace operator, therefore we can find $f\in\cun(\Sigma)$, unique up to an additive constant, such that $\kappa_g=\Delta f$ or equivalently $\kappa_{e^{-2f}g}=0$.
\end{proof}

In dimension $2$, the Newlander-Nirenberg theorem states any almost complex structure on a surface is integrable, since the Nijenhuis tensor vanishes identically. As a corollary of Theorem \ref{genus0} we obtain an elementary proof of this fact. For this, notice that an almost complex structure in dimension $2$ is precisely an orientation together with a conformal class. 

\begin{corollary}\label{holstr}
Let $\Sigma$ be an (open) oriented smooth surface with a conformal structure $[g]$. Then there exist a holomorphic atlas on $\Sigma$ compatible with the given conformal structure.
\end{corollary}
\begin{proof}
Fix $g\in[g]$ and $p\in\Sigma$. Choose a positively oriented chart $\phi$ near $p$, $\phi(p)=0$ and construct a metric $h$ on the torus $\tz:=\rz^2/\zz^2$ which coincides with the push-forward $\phi_* g$ near $0+\zz^2$, by using a partition of unity. By Theorem \ref{genus0}, $h$ is conformally flat, so there exists $f\in\cun(\tz)$ such that $e^{-2f}h$ is flat. Then the geodesic exponential map $\exp_{[0]}$ with respect to  $e^{-2f}h$ is a local isometry between $T_{[0]}\tz$ and $\tz$. It follows that $e^{-2\phi^*f}g$ is isometric near $p$ to the euclidean space $\rz^2$ with its standard metric. This isometry is \emph{conformal} with respect to the initial metric $g$. Since $p$ was arbitrary, we have found an atlas consisting of conformal maps. The changes of charts are positively oriented conformal maps between domains of $\rz^2=\cz$ with the euclidean metric, i.e., holomorphic.
\end{proof}

The idea of the above proof can be succinctly stated as follows: since by Theorem \ref{genus0} every torus is conformally flat, it follows that every surface is locally conformally flat.

\section{Singular flat surfaces}

Let $(r,\sigma)\in(0,\infty)\times [0,2\pi)$ be the polar coordinates on $\rz^2$. We distinguish the following flat singular metrics near $r=0$:
\begin{itemize}
 \item The \emph{small conical metric} of angle $\alpha>0$:
\[dr^2 + \left(\frac{\alpha}{2\pi}\right)^2 r^2d\sigma^2.\]
\item The \emph{cylindrical metric} of circumference $l$, also called conical of angle $0$:
\[\frac{dr^2}{r^2}+\left(\frac{l}{2\pi}\right)^2d\sigma^2.\]
\item The \emph{large conical metric} of angle $\alpha<0$:
\[\frac{dr^2}{r^4} + \left(\frac{\alpha}{2\pi}\right)^2 \frac{d\sigma^2}{r^2}.\]
\end{itemize}
The conical metric of negative angle can be seen near $r=\infty$ (i.e., in a neighborhood of $\infty$ in $\cz P^1$) by changing variables $r\mapsto \frac{1}{r}$ as
\[dr^2+\left(\frac{\alpha}{2\pi}\right)^2 r^2d\sigma^2.\]
These types of singularities can be conveniently described simultaneously as follows: recall that the standard metric on $\rz^2$ in polar coordinates takes the form $dr^2+r^2d\sigma^2$ (in particular the conical metric of angle $\alpha=2\pi$ is in fact non-singular).
\begin{lemma}\label{mcon}
Let $\theta\in \rz$. The singular flat metric near $r=0$
\[\frac{dr^2+r^2 d\sigma^2}{r^{2\theta}}\]
is conical of angle $2\pi(1-\theta)$, and moreover of circumference $2\pi$ if $\theta=1$.
\end{lemma}
The (immediate) proof consists of re-writing the metric in the variable 
\[
R=\begin{cases}
\frac{r^{1-\theta}}{1-\theta}&\text{if $\theta<1$};\\
r,&\text{if $\theta=1$};\\
\frac{\theta-1}{r^{1-\theta}}&\text{if $\theta>1$}.
\end{cases}
\]

\section{Uniformization of the sphere and of the projective plane}

It is a classical fact, which we wish to partly clarify in this note, that in any given conformal class on a surface $\Sigma$ there exist metrics of constant Gaussian curvature, of sign depending on the Euler characteristic of $\Sigma$. The case of $\chi(\Sigma)=0$ was explained in Theorem 
\ref{genus0}. For $\Sigma$ of negative Euler characteristic the hyperbolic metric is moreover unique in each conformal class. We will not prove this here. Instead we will give a self-contained treatment of the positive case. 

Let first $\Sigma$ be a sphere with a conformal class $[g]$. Note that, in contrast to Teichm\"uller theory, a spherical metric in the conformal class $[g]$ is by no means unique, in fact such metrics are parametrized by the cosets
\[\mathrm{PSL}_2(\cz)/\mathrm{SO}_3\]
of conformal transformations of the standard sphere $\cpu$ modulo isometries, which form a real manifold of dimension $3$. This non-uniqueness is the main reason of failure of variational methods in proving the existence of spherical metrics.

Fix a smooth cut-off function $\phi:[0,\infty)\to [0,\infty)$ with compact support in $[0,\epsilon/2]$ and which is $1$ on $[0,\epsilon/4]$. 
\begin{lemma}\label{dlog}
For any $\phi$ as above, the function $\Delta(\phi(r)\log(r))\in\cun_c(\rz^2\setminus\{0\})$
has compact support in the $[\epsilon/4,\epsilon/2]$-annulus, and satisfies
\[\int_\Sigma \Delta(\phi(r)\log(r))  \mu_{\rz^2}=2\pi\]
where $\Delta$ is the flat Laplacian on $\rz^2$.
\end{lemma}
\begin{proof}
In polar coordinates, $\Delta=-r^{-1}\pr r\pr-r^{-2}\partial_\sigma^2$ and the volume form is $rdrd\sigma$. The assertion on the support holds evidently. Thus 
\begin{align*}
\int_\Sigma \Delta(\phi(r)\log(r)) |\mu_g|=&\int_0^{2\pi}d\sigma\int_0^\infty (-r^{-1}\pr r\pr)(\phi(r)\log(r))
r dr\\=&-2\pi\int_0^\infty \pr r\pr(\phi(r)\log(r))dr\\
=&-2\pi r\pr(\phi(r)\log(r))|^\infty_0\\
=&2\pi.\qedhere 
\end{align*}
\end{proof}

\begin{theorem} \label{unifsph}
Let $\Sigma$ be a topological sphere with a smooth structure and a conformal class $[g]$. There exist in $[g]$
\begin{itemize}
\item Euclidean metrics, i.e., complete flat metrics outside one marked point $p$ with cone angle $-2\pi$;
\item Spherical metrics.
\end{itemize}
\end{theorem}
\begin{proof} 
Since $g$ is locally conformally flat, we may assume without loss of generality that $g$ is flat near a fixed point $p$. Let $(x,y)$ be flat orthogonal coordinates near $p=(0,0)$, and $(r,\sigma)\in(0,\epsilon)\times[0,2\pi)$ polar coordinates, such that $g=dr^2+r^2d\sigma^2$ near $r=0$.
By the Gauss-Bonnet formula, $\int_\Sigma\kappa_g\mu_g=4\pi$. Using Lemma \ref{dlog}, we deduce
\[\int_\Sigma(\kappa_g-2\Delta(\phi(r)\log r)))\mu_g=0.\]
The function $\kappa_g-2\Delta(\phi(r)\log r))$ is smooth with compact support on $\Sigma\setminus\{p\}$, in particular it is smooth on $\Sigma$. Any function of mean zero on a closed Riemannian manifold lives in the image of the Laplacian, i.e., there exists $f\in\cun(\Sigma)$ with $\Delta f=\kappa_g-2\Delta(\phi(r)\log r)$, or equivalently by \eqref{ccc}, such that the metric $g':=e^{-2f-4\phi(r)\log r}g$ is flat on its domain of definition $\Sigma\setminus\{p\}$.
Near $p=0$, we have by setting $R:=r^{-1}$:
\[g'=e^{-2f}\frac{dr^2+r^2d\sigma^2}{r^4}=e^{-2f}(dR^2+R^2d\sigma^2).\]
The conformal factor $f$ is smooth on $\Sigma$, so $e^{-2f}$ is bounded above and below by positive constants.
It follows that near the singularity, the flat metric $g'$ is quasi-isometric to the Euclidean metric
$dR^2+R^2d\sigma^2$, so in particular $g'$ is \emph{complete}. Since $\Sigma$ is by hypothesis a sphere, it follows that $\Sigma\setminus\{p\}$ is simply-connected. Therefore $\Sigma\setminus\{p\}$ is isometric to any of its tangent planes (which we identify with $\rz^2$) by the geodesic exponential map, denoted $\Phi$. This proves the first statement.

The isometry
$\Phi:\rz^2\to\Sigma\setminus\{p\}$ interchanges compact sets, it is therefore continuous, together with its inverse, at the point at infinity in the $1$-point compactifications of the two spaces, which are $\cpu$, respectively the initial surface $\Sigma$.

Now observe that $\Phi$ is an isometry with respect to the metric $g'$, it is therefore \emph{conformal}
with respect to $g$ outside $p$, or equivalently holomorphic with respect to the holomorphic structure from Corollary \ref{holstr}.

In conclusion, we have constructed a homeomorphism $\Phi:\cpu\to\Sigma$ which is smooth and bi-holomorphic
outside a point $p$. The singularity at $p$ is removable by continuity, thus $\Phi$ is globally bi-holomorphic. 
The desired spherical metric is obtained by push-forward of the standard metric on $\cpu$ via $\Phi$.
\end{proof}
The first statement is more precise - albeit in a particular case - than the corresponding statement from Theorem \ref{unifcone}, since we do not get here any translation component in the cone of angle $-2\pi$.

A similar approach works for $\rpd$:

\begin{theorem}
Let $\Sigma$ be a surface homeomorphic to $\rpd$ with a smooth structure and a conformal class $[g]$. There exist in $[g]$
\begin{itemize}
\item Cylindrical metrics, i.e., complete flat metrics outside one marked point $p$ with cone angle $0$;
\item Spherical metrics.
\end{itemize}
\end{theorem}
\begin{proof}
We will use the notation and ideas of Theorem \ref{unifsph}.

We choose a metric $g\in[g]$ which is flat near $p$ and flat complex coordinates $z$ so that $g=|dz|^2$ near $p$. By Gauss-Bonnet and Lemma \ref{dlog}, the smooth function on $\Sigma$ defined by $\kappa_g-\Delta(\phi(r)\log r)$ is of zero mean, hence it lives in the image of $\Delta$, say $\Delta f=\kappa_g-\Delta(\phi(r)\log r)$. Thus by \eqref{ccc}, the metric $e^{-2f-2\phi(r)\log r}g$
is flat on the topological M\"obius band $\Sigma\setminus \{p\}$. This metric is quasi-isometric to $r^{-2}|dz|^2$ hence it is complete. It follows that its universal cover is isometric to a quotient of $\rz^2$ by some non-orientable infinite cyclic group of isometries. Such a group is conjugated to a group generated by $\gamma: z\mapsto \bar{z}+a$, $a\in \rz^*_+$. 

The same reasoning holds for the standard metric on $\rpd$, namely $\rpd\setminus\{p_0\}$ is isometric to $\rz^2/\langle \gamma_0\rangle$ for $\gamma_0: z\mapsto \bar{z}+a_0$, $a_0\in \rz^*_+$; after multiplying one of the two metrics by a constant, it follows that $\Sigma\setminus \{p\}$ and $\rpd\setminus\{p_0\}$ are isometric. Any choice of isometry $\Phi:\rpd\setminus\{p_0\}\to
\Sigma\setminus \{p\}$ will preserve closed and bounded sets (i.e., neighborhoods of $p, p_0$) i.e., $\Phi$ extends to the one-point compactifications as a homeomorphism from $\Sigma$ to $\rpd$. Moreover, $\Phi$ is conformal from the spherical metric on $\rpd\setminus\{p_0\}$ to $(\Sigma\setminus \{p\},g)$, so we deduce that $p$ is a removable singularity and $\Phi$ is a global conformal map from $\rpd$ to $\Sigma$, in particular smooth. The push-forward of the standard metric from $\rpd$ to $\Sigma$ is the desired spherical metric.
\end{proof}

\section{Conformally conical metrics}

To classify flat singular metrics we extend slightly the notion of conical metric of integer negative angle. A partial isometry of a metric space $\Sigma$ is an isometry from a subset of $\Sigma$ onto its image in $\Sigma$.

\begin{definition}
Let $l>0$ and $\theta\in \zz$, $\theta\geq 2$. On the universal cover of $\cz^*$ consider the deck transformation $\gamma$ corresponding to the generator $S^1$ of $\pi_1(\cz^*)$, and the partial isometry $\tau_l$ given by the translation with $l\in\rz$. A metric on a surface $\Sigma$ singular near $p$ is called \emph{conical of angle $2\pi(1-\theta)$ and translation component $l$} if it is isometric near $p$ with the quotient of $\tilde{\cz^*}$ by the partial isometry $\tau_{l}\gamma^{1-\theta}$.
\end{definition}
Geometrically, we can construct such a metric by considering the branched Riemann cover of the function $z^{1-\theta}$, cut along a pre-image of the positive half-line and glue back after translating by $l$. The resulting surface is well-defined near infinity, where the translation makes sense.

\begin{theorem}\label{confharmo}
Let $f$ be a real harmonic function in a neighborhood of $0\in\rz^2$, and $g$ a conical metric on $\rz^2$ near $0$, of angle $\alpha=2\pi(1-\theta)$.
\begin{itemize}
\item If the angle $\alpha$ is positive, or not an integer multiple of $2\pi$, then $e^{-2f}g$ is also conical of angle $\alpha$.
\item If $\alpha=0$ and $g$ is cylindrical of circumference $l$, then $e^{-2f}g$ is also cylindrical of circumference $e^{-f(0)}l$.
\item If $\alpha$ is a negative multiple of $2\pi$, then $e^{-2f}g$ is also conical of angle $\alpha$ and translation $l$ for some explicit $l$.
\end{itemize}
\end{theorem}
\begin{proof}
If $\alpha=2\pi$ (i.e., $\theta=0$) there is nothing to prove, since any smooth flat metric is conical of angle $2\pi$. Assume therefore $\theta\neq 0$.
Let $F(z)$ be the unique holomorphic function with real part $f$, and imaginary part vanishing at $0$.
Note that $e^{-2f}=|e^{-F}|^2$. 

\subsection*{Case 1.} Assume $\theta\notin \nz^*=\{1,2,\ldots\}$. By Lemma \ref{mcon} the conical metric takes the form $\frac{dr^2+r^2 d\sigma^2}{r^{2\theta}}=|z^{-\theta}dz|^2$, so
\[e^{-2f}g=|e^{-F}z^{-\theta}dz|^2.\]
We would like to find a new complex coordinate $u$ near $z=0$ so that
$e^{-2f}g=|u^{-\theta}du|^2$. We look for $u$ of the form $u=zh(z)$ 
with $h(z)$ holomorphic near $0$, $h(0)\neq 0$ so that 
\[e^{-F}z^{-\theta}dz=u^{-\theta}du.\] 
Writing $e^{-F(z)}=\sum_{n=0}^\infty a_n z^n$, and $v(z):=h^{1-\theta}(z)=\sum_{n=0}^\infty v_nz^n$ a simple calculation shows that the coefficients of $v$ are uniquely determined by those of $e^{-F}$:
\[v+\frac{z}{1-\theta} v'=e^{-F(z)},\]
or equivalently $v_n=\frac{1-\theta}{1-\theta+n} a_n$. The resulting Taylor series for $v$ is clearly convergent for small $|z|$ since $F$ is holomorphic near $0$. Moreover $v_0\neq 0$ since $a_0=e^{-F(0)}\neq 0$, thus there exists
$h:=e^{\frac{\log(v)}{1-\theta}}$ with the property that $g=|u^{-\theta}du|^2$ is conical of angle $2\pi(1-\theta)$.

\subsection*{Case 2.} Assume $\theta=1$. The conical metric of angle $0$ (i.e., cylindrical) and circumference $l$ is 
\[g=\left(\tfrac{l}{2\pi}\right)^2\left|\frac{dz}{z}\right|^2. 
\]
We want to find $u(z)$ near $z=0$ so that 
\[e^{-2f}g=e^{-2f(0))}\left(\tfrac{l}{2\pi}\right)^2\left|\frac{du}{u}\right|^2.\]
We proceed as above by solving $e^{-F(z)}z^{-1}dz=e^{-F(0)} u^{-1}du$ for $u(z)=z e^{v(z)}$, or equivalently 
\[e^{F(0)-F(z)}=1+zv'(z).\]
This equation determines uniquely a convergent series for $v(z)$ with $v(0)=0$.

\subsection*{Case 3.} Assume now $\zz\ni\theta\geq 2$, thus the angle is a negative integer multiple of $2\pi$.
We claim that we can find a complex coordinate $u$ near $0$ so that
\[e^{-2f}g=\left|\frac{(1+cu^{\theta-1})du}{u^\theta}\right|^2.\]
for some constant $c$ depending on $f$. The strategy is to solve instead
\[e^{-F(z)}\frac{dz}{z^\theta}=\frac{(1+cu^{\theta-1})du}{u^\theta}\]
with $u=zh^{\frac{1}{1-\theta}}(z)$ for some $h(z)$ holomorphic, $h(0)\neq 0$. Setting $e^{-F(z)}=\sum_{n=0}^\infty a_n z^n$, the equation to solve is
\[\sum_{\theta-1\neq n\geq 0} \frac{a_n}{n+1-\theta} z^{n+1-\theta} + a_{\theta-1}\log z= \frac{z^{1-\theta}h(z)}{1-\theta} +c\log z +\frac{c}{1-\theta}\log h.
\]
This equation has a solution only if $c=a_{\theta-1}=\Res(z^{-\theta}e^{-F(z)})$. With this value for $c$, we would like to know that the solution $h$ is holomorphic. By denoting $A(z)$ the left-hand side times $(1-\theta)z^{1-\theta}$, we write
\[h(z)+cz^{\theta-1}\log h(z)= A(z).\]
Note that $A(0)=e^{-F(0)}\neq 0$.
The holomorphic function $\Phi(h,z):=h+cz^{\theta-1}\log h-A(z)$ defined in a neighborhood of $h=A(0),z=0$ has the property 
\begin{align*}
\Phi(A(0),0)=0,&&\frac{\partial\Phi(A(0),0)}{\partial h}=1,
\end{align*}
therefore by the implicit function theorem there exists $h(z)$ holomorphic
in $z$ such that $\Phi(h(z),z)=0$. 
The third case of the theorem follows from the next Lemma. The translation component is therefore
$2\pi |\Res(z^{-\theta}e^{-F(z)})|$, where $F$ is the holomorphic function of real part $f$.
\end{proof}

\begin{lemma}
Let $c\in\cz$. The flat metric $g_c=\left|\frac{(1+cu^{\theta-1})du}{u^\theta}\right|^2$ is conical near $u=0$ of angle $2\pi(1-\theta)$ and translation component $2\pi|c|$.
\end{lemma}
\begin{proof}
By changing variables $z=\alpha u$ with $\alpha\in\cz, |\alpha|=1$ we see that $g_c$ is isometric to $g_{\alpha^{\theta-1}c}$, hence we may assume $c=il$, $l\in\rz$, $l>0$. Let $X:=\{z\in\cz^*;|z|<\epsilon\}$.
Starting from the half-line $z>0$ define $v:=\frac{1}{1-\theta}z^\frac{1}{1-\theta}+il\log z$ using the principal branch of the logarithm. Since $g=|dv^2|$, it follows that for $\arg(z)\leq 2\pi$, $v$ is an isometry from $(\tilde{X},g_{\theta,c})$ to a neighborhood of infinity in the universal cover $\tilde{\cz}^*$ with its standard flat metric. Moreover, for $z\in \tilde{X}$ we clearly have $v(\gamma z)=\rho^{\theta-1}v(x)-2\pi l$ for $\gamma$ the generator of the deck transformation group of $\tilde{X}\to X$ (i.e., the lift of rotation by $2\pi$), and $\rho$ the corresponding generator for $\tilde{\cz}^*\to\cz^*$. We see that $\gamma$ is conjugated to $u\mapsto \rho^{\theta-1}u-2\pi l$ via the partial isometry $v$. Thus $X$ is isometric to the flat surface obtained from $\tilde{\cz}^*$ by identifying two half-rays in the pre-image of the positive real axis in $\tilde{\cz}^*$ encompassing a sector of angle $2\pi(\theta-1)$ via the translation by $2\pi l$ (since for $c\neq 0$ this operation only makes sense near infinity, we get an incomplete metric).
\end{proof}

\section{Uniformization by flat singular metrics}

Let $\Sigma$ be a closed surface with a fixed conformal class $[g]$. 
Choose points $p_1,\ldots, p_N\in \Sigma$ and real numbers $\alpha_j$, $1\leq j\leq N$.
\begin{theorem}\label{unifcone}
There exist on $Y:=\Sigma\setminus \{p_1,\ldots, p_N\}$ a flat metric in the conformal class of $g$
with conical singularities of angles $\alpha_j$ at $p_j$, $1\leq j\leq N$ if and only if
\begin{equation}\label{condgb}
\sum_{j=1}^N \alpha_j=2\pi N-2\pi \chi(\Sigma).
\end{equation}
If $\alpha_j\in -2\pi \nz^*$, the conical singularity at $p_j$ may have a translation component. The flat metric is unique inside its conformal class up to homothety.
\end{theorem}
\begin{remark}
For $\alpha_j$ positive and $\Sigma$ orientable the result appears in Troyanov \cite{Troy}.
\end{remark}

\begin{proof}
The necessity of \eqref{condgb} follows from Gauss-Bonnet, since by Lemma \ref{dlog} every conical point $p_j$ contributes $2\pi\theta_j$ to the total curvature. Uniqueness up to homothety is a consequence of the fact that harmonic functions on $\Sigma$ must be constant.

Since any surface is locally conformally flat, we may assume that $g$ is flat in an $\epsilon$-neighborhood of the marked points $p_j$. Near each $p_j$ consider flat coordinates given for instance by the geodesic exponential map, and then polar coordinates $r_j\in[0,\epsilon),\sigma_j\in[0,2\pi)$. Recall that $\Delta$ is the Laplacian of the metric $g$ on $\Sigma$.
It follows from the Gauss-Bonnet formula and Lemma \ref{dlog} that for any reals $\theta_j$, if we set $\Theta:=\sum_{j=1}^N \theta_j\phi(r)\log(r)$, we have
\[\int_\Sigma(\kappa_g -\Delta\Theta)|\mu_g|=2\pi\left(\chi(\Sigma)-\sum_{j=1}^N\theta_j\right).\]
If we choose $\theta_j$ so that $\alpha_j=2\pi(1-\theta_j)$ then condition \eqref{condgb} is equivalent to 
\begin{equation}\label{condthet}
\sum_{j=1}^N\theta_j=\chi(\Sigma).
\end{equation}
Thus with this choice, the smooth function $\kappa_g -\Delta\Theta$ on $\Sigma$ has zero mean, therefore it lives in the image of the Laplace operator of $g$. In other words, there exists $f\in\cun(\Sigma)$ with
\begin{equation}\label{kft}
\kappa_g=\Delta(f+\Theta).
\end{equation}
This is equivalent to the flatness of the metric $g':=e^{-2f-2\Theta}g$ on $Y=\Sigma\setminus\{p_1,\ldots,p_N\}$.

We claim that $g'$ is conical with angles $2\pi(1-\theta_j)$ at $p_j$. In polar coordinates near $p_j$ the initial metric is $dr^2+r^2d\sigma^2$, while $\Theta=\theta_j\log(r)$ for $r<\epsilon/4$.
This means that near $p_j$, the metric $g'$ takes the form 
\[g'=e^{-2f-2\theta_j\log r} (dr^2+r^2d\sigma^2)=e^{-2f}\left(\frac{dr^2+r^2d\sigma^2}{r^{2\theta_j}}\right).\]
From \eqref{kft}, the function $f$ is harmonic near $p_j$, while from lemma \ref{mcon} the metric inside the bracket is conical of angle $2\pi(1-\theta_j)$. From Theorem \ref{confharmo} it follows that $g'$ is conical of angle $2\pi(1-\theta_j)$, with possible translation component if $\theta_j$ is an integer greater than or equal to $2$.
\end{proof}

It is not clear from the proof if translations do appear in the cones, or are simply a byproduct of our method. To prove their existence we give an argument based on complex analysis.

\section{Link with Abelian differentials}

Assume $\Sigma$ is orientable, hence a compact complex curve.

For $\theta_j$ integers satisfying \eqref{condthet}, the conclusion of Theorem \ref{unifcone}
follows if we can find a meromorphic $1$-form
$\mu\in\cM(\Sigma,\Lambda^{1,0}(\Sigma))$ (with respect to the holomorphic structure defined by our fixed conformal structure) with prescribed zeros or poles of order $\theta_j$ at $p_j$. Then the symmetric tensor product 
\[|\mu|^2:=\tfrac12 (\mu\otimes\bar{\mu}+\bar{\mu}\otimes\mu)\]
is a flat conical metric on $\Sigma\setminus\{p_1,\ldots,p_N\}$ with angles $2\pi(1-\theta_j)$ at $p_j$. 
More generally, the absolute value of a meromorphic quadratic differential $A\in\cM(\Sigma,\Lambda^{1,0}(\Sigma) \otimes \Lambda^{1,0}(\Sigma))$ is a flat conical metric with angles integer multiples of $\pi$. Zeros or poles of order $m$ correspond to cone angles $\pi(2-m)$. These facts are consequences of Theorem \ref{confharmo} since for any meromorphic function $\nu(z)$, the real function $\log|\nu(z)|$ is harmonic outside its zeros and poles.

For the flat conical metric $|\mu|^2$ defined by any Abelian differential $\mu$, the translation component in the cones of angle a negative integer multiple of $2\pi$, respectively the circumference of the cylinders, are given by the absolute value of the residues of $\mu$ at the singular points. These residues are generically nonzero for meromorphic Abelian differential of third kind.

\subsection*{Acknowledgements} I am indebted to C.~Guillarmou, C.~Joita, A.~Moroianu and J.-M.~Schlenker for discussions on the subject of this paper. The hospitality of the IHES at Bures-sur-Yvette is gratefully acknowledged.

\end{document}